\documentclass[10pt]{article}
\font\smallit=cmti10

\usepackage{amssymb,latexsym,amsmath,epsfig,amsthm} %% Add other packages as necessary

\makeatletter

\renewcommand\section{\@startsection {section}{1}{\z@}
{-30pt \@plus -1ex \@minus -.2ex}
{2.3ex \@plus.2ex}
{\normalfont\normalsize\bfseries\boldmath}}

\renewcommand\subsection{\@startsection{subsection}{2}{\z@}
{-3.25ex\@plus -1ex \@minus -.2ex}
{1.5ex \@plus .2ex}
{\normalfont\normalsize\bfseries\boldmath}}

\renewcommand{\@seccntformat}[1]{\csname the#1\endcsname. }

\makeatother

\theoremstyle{definition}

\usepackage{amsfonts}
\usepackage{amssymb}
\usepackage{latexsym}
\usepackage{subfigure}
\usepackage{amsmath}
\usepackage{amsthm}
\usepackage{xcolor}
\usepackage{lineno}

\newtheorem{thm}{Theorem}[section]

\newtheorem{lem}[thm]{Lemma}
\newtheorem{cor}[thm]{Corollary}
\newtheorem{prop}[thm]{Proposition}

\newtheorem{prob}[thm]{Problem}

\newtheorem{rem}[thm]{Remark}

\newcommand{\Z}{\mathbb{Z}}

\newcommand{\abegroup}{\mathbb{Z}_{n_1}\oplus \cdots \oplus \mathbb{Z}_{n_r}} % notation for finite abelian group
\newcommand{\abeliangroup}[1]{\mathbb{Z}_{n_1}\oplus \cdots \oplus \mathbb{Z}_{n_{#1}}} % notation for finite abelian group
\newcommand{\pgroup}[1]{\mathbb{Z}_{p^{\alpha_1}} \oplus  \cdots \oplus \mathbb{Z}_{p^{\alpha_{#1}}}} % notation for p-group

\newcommand{\evector}[1]{(\varepsilon_1, \ldots , \varepsilon_{#1})} % an varepsilon or idempotent vector of length #1
\newcommand{\pvector}[1]{(p^{\alpha_1}, \ldots, p^{\alpha_{#1}})}
\newcommand{\Dav}[1]{\mathsf{D}(#1)} % notation for the Davenport constant of a group
\newcommand{\D}[3]{\mathsf{D}({#3},{#2})} % D is for Davenport #1 is the function (which we have abandoned), #2 is m, #3 is the group

\newcommand{\elem}[2]{e_{#1}({#2})} % for the elementary symmetric polynomial  -  #2 is a vector
\newcommand{\Elem}[2]{e_{#1}(x_1, \ldots , x_{#2})} % for the elementary symmetric  polynomial - #1 counts the number of variables
\newcommand{\Pow}[2]{p_{#1}(x_1, \ldots, x_{#2})}  % the power sum polynomial - #1 counts the number of variables
\newcommand{\Up}[2]{U({#1},{#2})} % defined by us to apply Baker-Schmidt to obtain Upper bound (U is for upper)
\newcommand{\Low}[2]{L({#1},{#2})} % defined by us to obtain Lower bound (L is for lower)
\newcommand{\Odd}[2]{E({#1},{#2},odd)}
\newcommand{\Even}[2]{E({#1},{#2},even)}
\newcommand{\floor}[1]{\lfloor {#1} \rfloor} % to take the floor of a number

\newcommand{\dom}[1]{t({#1})} % for the size of the smallest dominating set for when we seek to apply Newton-Girard Formulae

\makeatletter
\def\imod#1{\allowbreak\mkern10mu({\operator@font mod}\,\,#1)}
\makeatother

\begin{document}

\begin{center}
\uppercase{\bf Higher Degree Davenport Constants over Finite Commutative Rings}
\vskip 20pt
{\bf Yair Caro}\\
{\smallit Department of Mathematics, University of Haifa-Oranim, Israel}\\
{\tt yacaro@kvgeva.org.il}\\ 
\vskip 10pt
{\bf Benjamin Girard}\\
{\smallit Sorbonne Universit\'e and Universit\'e de Paris, CNRS, Institut de Math\'ematiques de Jussieu - Paris Rive Gauche, Paris, France}\\
{\tt benjamin.girard@imj-prg.fr}\\ 
\vskip 10pt
{\bf John R. Schmitt}\\
{\smallit Department of Mathematics, Middlebury College, Middlebury, Vermont, USA}\\
{\tt jschmitt@middlebury.edu}

\end{center}
\vskip 20pt

%\centerline{\smallit Received: , Revised: , Accepted: , Published: } % We will fill in the dates
\vskip 30pt

\centerline{\bf Abstract}

\noindent

We generalize the notion of Davenport constants to a `higher degree' and obtain various lower and upper bounds, which are sometimes exact as is the case for certain finite commutative rings of prime power cardinality.  Two simple examples that capture the essence of these higher degree  Davenport constants are the following.  1) Suppose $n = 2^k$, then every sequence of integers $S$ of length $2n$  contains a subsequence $S'$ of length at least two such that $\sum_{a_i,a_j \in S'} a_ia_j \equiv 0 \pmod{n}$ and the bound is sharp.  2) Suppose $n \equiv1 \pmod{2}$, then every sequence of integers $S$ of length $2n -1$  contains a subsequence $S'$ of length at least two such that  $\sum_{a_i,a_j \in S'} a_ia_j \equiv 0 \pmod{n}$.  These examples illustrate that if a sequence of elements from a finite commutative ring is long enough, certain symmetric expressions have to vanish on the elements of a subsequence.

\pagestyle{myheadings}
%\markright{\smalltt INTEGERS: 21 (2021)\hfill}
\thispagestyle{empty}
\baselineskip=12.875pt
\vskip 30pt

%\begin{small}
%{\bf Keywords:} Davenport constant, zero-sum, elementary symmetric polynomial\\
%\indent {\bf MSC:} 05E40 (Primary), 11B30, 13M10 (Secondary).
%\end{small}

\section{Introduction}\label{sec:introduction}

Throughout this paper, let $p$ denote a prime number and $q=p^{\alpha}$ a prime power.

Let $G$ be a finite abelian group.  A finite sequence $S=(g_1, \ldots, g_{\ell})$ of elements of $G$ is called a {\it sequence over $G$}, where order is disregarded and repetition is allowed.  Its {\it length}, denoted $|S|$, is the number of elements therein, counted with multiplicity.  A sequence of $G$ is said to be {\it zero-sum} if the sum of its elements is zero in $G$.  A sequence $S$ of $G$ is said to be {\it zero-sum free} if every non-trivial subsequence of $S$ has sum different to zero.  For a group $G$, the {\it Davenport constant of $G$}, which we denote by $\Dav{G}$, is the smallest positive integer $t$ such that every sequence $S$ over $G$ of length $|S| \geq t$ contains a non-empty zero-sum subsequence.  That is, we seek the smallest $t$ for which there is a non-trivial solution of
\begin{equation*}\label{eqn:Davenport}
\varepsilon_1g_1+ \cdots +\varepsilon_tg_t=0,
\end{equation*}
where each $\varepsilon_i$ is $0$ or $1$.

Study of this number intensified in the 1960s with K. Rogers \cite{Rogers63} in 1963, and later with H. Davenport in 1966 as explained by J.E. Olson in \cite{Olson69a} and has continued unabated since; see, for example, a useful survey by W. Gao and A. Geroldinger \cite{Gao-Geroldinger06}.

The cyclic group with $n$ elements will be denoted $\Z_n$.  Further, it is well-known that by the Fundamental Theorem of Finite Abelian Groups that for any finite non-trivial abelian group $G$ there exist integers $n_1, \ldots , n_r$ where $1 < n_1 \mid \ldots \mid n_r$ so that $G$ can be written uniquely as $$G \cong \abegroup.$$
The integer $r$ is called the {\it rank} of $G$ and denoted $r(G)$.  We use $\mathsf{d}^*(G)$ to denote the value $\sum_{i=1}^r(n_i-1)$.

The value of $\Dav{G}$ was determined independently by J.E.~Olson \cite{Olson69a} and D. Kruyswijk \cite{vanEmdeBoasKruyswijk67} when $G$ is a $p$-group, and by J.E.~Olson \cite{Olson69b} when $G$ has rank at most $2$.

\begin{thm}\label{thm:Olson}{[J.E.~Olson \cite{Olson69a}, \cite{Olson69b}, and D.~Kruyswijk \cite{vanEmdeBoasKruyswijk67}]}
If $G$ is a $p$-group or $r(G) \leq 2$, then $\Dav{G} = 1+\mathsf{d}^*(G)$.
\end{thm}

The value of $\Dav{G}$ is unknown in general.  For a survey of results, see the work of A.~Geroldinger \cite{Geroldinger09} and the work of A.~Geroldinger and F.~Halter-Koch \cite{GeroldingerHalter-Koch06}.  Recently, B.~Girard \cite{Girard18} has shown that for all integers $r \geq 1$, $\Dav{\Z_n^r}\sim rn$ as $n \rightarrow \infty$. 

We now introduce our object of study.  Let $(A,+,\cdot)$ be a finite commutative ring. For any positive integer $m$ and any sequence $S=(a_1,\dots,a_\ell)$ over $A$, we set
$$e_m(S):=\displaystyle\sum_{1 \le i_1 < \dots < i_m \le \ell} \displaystyle\prod^m_{j=1} a_{i_j}.$$
We say that $S$ is an \em $m$-zero sequence \em whenever $e_m(S)=0$, and that it is an \em $m$-zero free sequence \em whenever, for every subsequence $S'$ of $S$ such that $|S'| \ge m$, one has $e_m(S') \neq 0$. 
We denote by $\D{}{m}{A}$ the smallest positive integer $t$ such that every sequence $S$ over $A$ of length $|S| \ge t$ contains a subsequence $S'$ of length $|S'| \ge m$ for which $e_m(S')=0$.

Notice that when $m=1$ we recover the classical Davenport constant discussed above.  As a result, we may consider $\D{}{m}{A}$ as the {\em $m^{th}$-degree Davenport constant}.

In this paper we examine this higher degree Davenport constant.  This line of investigation that we follow is suggested by the work of A.~Bialostocki and T.D.~Luong \cite{BialostockiLuong09}, \cite{BialostockiLuong14}, and T.~Ahmed, A.~Bialostocki, T. Pham and Le Anh Vinh \cite{AhmedBialostockiPhamLe19}.

We proceed as follows. In Section \ref{section:Baker-Schmidt} we examine the higher degree Davenport constant in the case that $A=\Z_n$.  Of particular use is a result of R. Baker and W. Schmidt \cite{Baker-Schmidt80} (and see also \cite{Baker-Schmidt81}).  We obtain a precise result in the case that $n$ is a prime power and $m$ is power of the same prime.   In Section \ref{section:pgroups} we give an upper bound for the higher degree Davenport constant in the case that $A$ is of the form $\pgroup{r}$ and a lower bound for any $A$ of the form $\abeliangroup{b}$ from which we deduce a sharp value of the higher Davenport constant for rings of the form $\pgroup{r}$ when $m$ is also a power of $p$.  In Section \ref{section:Girard-Newton} we show how to use the classical Girard-Newton formulae, which allow one to express the elementary symmetric polynomial of degree $k$ by a combination of power sum polynomials, to obtain upper bounds.  In Section \ref{section:conclusion} we present some open problems.

\section{Bounds for cyclic groups}\label{section:Baker-Schmidt}

First, we note an easy lower bound on $\D{\elem{m}{\bf x}}{m}{\Z_n}$.  Consider the sequence ${\bf 1}:=(1,\ldots ,1)$ of length $t$.  If $t=m$, then the only subsequence of length at least $m$ is the given sequence itself and $\elem{m}{\bf 1}=1 \not \equiv 0 \pmod{n}$.  Further, suppose that for each $\ell$ with $t > \ell \geq m$ we have ${\ell \choose m} \not \equiv 0 \pmod{n}$.  Then there exists no subsequence of ${\bf 1}$ of length at least $m$ which evaluates to zero modulo $n$.  Thus, we define $\Low{n}{m}$ to be the smallest integer $t \geq m+1$ such that ${t \choose m} \equiv 0 \pmod{n}$.  We have

\begin{equation}\label{eqn:L1sequence}
\D{\elem{m}{\bf x}}{m}{\Z_n} \geq \Low{n}{m}.
\end{equation}

Throughout the remainder of this section, let $n=p^{r}=q$.

An $s$-tuple $\evector{s}$ with each $\varepsilon_i=0$ or $1$ will be called {\it idempotent}.  Whenever $\varepsilon_1+\cdots + \varepsilon_s$ is even (respectively, odd), an idempotent $s$-tuple will be called {\it even} (respectively, {\it odd}). Further, for a fixed $m$, an idempotent $s$-tuple will be called ${\it m-artificial}$ (or just {\it artificial} when $m$ is clear) whenever $\varepsilon_1+\cdots + \varepsilon_s \leq m-1$, i.e. the number of $\varepsilon_i$ that take on the value $1$ is strictly less than $m$.

We will apply the following theorem of R.C. Baker and W.M. Schmidt \cite{Baker-Schmidt80, Baker-Schmidt81}.

\begin{thm}\label{thm:BS2}[R.C. Baker, W.M. Schmidt \cite{Baker-Schmidt80, Baker-Schmidt81}] Suppose that ${\cal F}_1, \ldots , {\cal F}_{\ell}$ are polynomials in ${\bf x}=(x_1, \ldots , x_s)$ with coefficients in respective $p$-groups $G_1, \ldots , G_{\ell}$, and of respective degrees $d_1, \ldots , d_{\ell}$.  Write $A$ or $B$, respectively, for the number of even or the number of odd idempotent solutions of

\begin{equation*}
{\cal F}_1({\bf{\varepsilon}})=0, \ldots, {\cal F}_{\ell}({\bf{\varepsilon}})=0.
\end{equation*}

If 
\begin{equation*}
s > d_1(\Dav{G_1}-1)+\cdots +d_s(\Dav{G_{\ell}}-1),
\end{equation*}
then 
\begin{equation*}
A \equiv B \pmod{p}.
\end{equation*}
\end{thm}

To facilitate an application of Theorem \ref{thm:BS2} to our specific setting, we define for integers $q=p^{r}$ and $m$ the function $\Up{q}{m}$ to be the smallest integer $t \geq m(q-1)+1$ such that 

$$\sum_{0 \leq 2j \leq m-1} {t \choose 2j} \not \equiv \sum_{1 \leq 2j+1 \leq m-1} {t \choose 2j+1}  \pmod{p}. $$

Furthermore, for integers $n$ and $m$ we denote the set of all $m$-artificial idempotent $n$-tuples with $\varepsilon_1+\cdots + \varepsilon_n$ equal to an even (odd) integer by $\Even{n}{m}$ ($\Odd{n}{m}$).  Clearly, we have $$|\Even{n}{m}|= \sum_{0 \leq 2j \leq m-1} {t \choose 2j} ~~~\text{and}~~~ \newline |\Odd{n}{m}| = \sum_{1 \leq 2j+1 \leq m-1} {t \choose 2j+1}.$$

\begin{thm}\label{thm:cyclicbound}
Let $r$ be a non-negative integer, $p$ a prime, $q=p^r$ and $m \geq 1$.  We have $$\Low{q}{m} \leq \D{\elem{m}{\bf x}}{m}{\Z_q} \leq \Up{q}{m}.$$
\end{thm}

\begin{proof}The lower bound was established above.  We establish the upper bound.

For a sequence $S=(a_1, \ldots ,a_{\ell})$ as opposed to seeking subsequences $S'$ of length at least $m$ such that $e_m(S') \equiv 0 \pmod{q}$, we may seek idempotent solutions that are not $m$-artificial to the following polynomial equation,

$$ \displaystyle\sum_{1 \le i_1 < \dots < i_m \le \ell} \displaystyle\prod^m_{j=1} a_{i_j}x_{i_j} \equiv 0 \pmod{q}.$$

To prove the upper bound, consider this degree-$m$ polynomial equation when the number of variables is $\Up{q}{m} \geq m(q-1)+1$, i.e. $\ell \geq m(q-1)+1$.  

Clearly, all $m$-artificial idempotent ${\Up{q}{m}}$-tuples are solutions to this equation since each monomial of the polynomial is a product of $m$ variables (and so at least one variable in each monomial evaluates as $0$ and so each monomial evaluates as  $0$).  From these solutions, we know that the number of even idempotent solutions $A$ is at least $\Even{\Up{q}{m}}{m}$ and the number of odd idempotent solutions $B$ is at least $\Odd{\Up{q}{m}}{m}$.  By the definition of $\Up{q}{m}$, we have that $|\Even{\Up{q}{m}}{m}| \not \equiv |\Odd{\Up{q}{m}}{m}| \pmod{p}$.  Thus, by Theorem \ref{thm:BS2}, there exists an idempotent solution that is not $m$-artificial. 
\end{proof}

\subsection{Properties of $\Up{q}{m}$ and $\Low{q}{m}$}

The lower bound $\Low{q}{m}$ and upper bound $\Up{q}{m}$ provided in Theorem \ref{thm:cyclicbound} motivate us to a numerical understanding of these functions in order to make them effective.

We begin with an investigation of $\Up{q}{m}$.  

Using Pascal's Identity and induction, one may show that $$\sum_{0 \leq 2j \leq m-1} {t \choose 2j} - \sum_{1 \leq 2j+1 \leq m-1} {t \choose 2j+1} = (-1)^{m-1} {t-1 \choose m-1}.$$  Thus, an alternate definition of $\Up{q}{m}$ is the smallest integer  $t \geq m(q-1)+1$ such that ${t-1 \choose m-1} \not \equiv 0 \pmod{p}$.  The former definition naturally arises in the proof of Theorem \ref{thm:cyclicbound} while the latter we use below.

We recall some classical results in number theory from the 19th-century.

Let $p$ be a prime number and $n > 1$ an integer.  The {\it $p$-adic valuation of $n$}, denoted $\nu_p(n)$, is the exponent of $p$ in the canonical decomposition in prime numbers of $n$ (and if $p$ does not divide $n$, then $\nu_p(n)=0$).  The base-$p$ expansion of $n$ is written as such, $n=a_kp^k+a_{k-1}p^{k-1}+\cdots +a_1p+a_0$.  Let $s_p(n) = a_k+a_{k-1}+\cdots +a_1+a_0$.

\begin{thm}[A.-M.~Legendre, 1808 \cite{Legendre1808}]
Let $p$ be a prime and let $n$ be a positive integer.  Then

$$\nu_p(n!) = \frac{n-s_p(n)}{p-1}.$$
\end{thm}

Legendre's Theorem was used to establish the following.

\begin{thm}[E.~Kummer, 1852 \cite{Kummer1852}]
The $p-$adic valuation of the binomial coefficient ${n \choose m}$ is equal to the number of `carry-overs' when performing the addition in base $p$ of $n-m$ and $m$. 
\end{thm}

% Is the generalization of Kummer's Theorem to multinomial coefficients ever useful?

When one uses Legendre's Theorem to prove Kummer's Theorem, an intermediate step gives 

\begin{eqnarray}\label{eqn:Kummer}
\nu_p({n \choose m})  &= &\nu_p(n!)-\nu_p(m!)-\nu_p((n-m)!)\\
		& = & \frac{s_p(m)+s_p(n-m)-s_p(n)}{p-1}.\label{eqn:Kummer2}
\end{eqnarray}

We repeatedly use this identity in the proofs given below. \\

\begin{prop}\label{prop:Up}
For an integer $m \geq 1$, a prime $p$ and $q$ a power of $p$, we have the following.
\begin{enumerate}
\item $m(q-1)+1 \leq \Up{q}{m} \leq mq.$

\item For $p \geq 2m-1, \Up{q}{m} = m(q-1)+1$.
\item For $m \leq p \leq 2m-2, \Up{q}{m} = mq+m-p.$
\item For $p \geq m$, the roots of ${t-1 \choose m-1} \in \mathbb{Z}_p[t]$ are $1, 2, \ldots , m-1.$ 

\end{enumerate}
\end{prop}

\begin{proof}
\begin{enumerate}
\item  The lower bound is by the definition.  Now assume that $\Up{q}{m} > m(q-1)+1$.  Consider the largest integer $T \geq m(q-1)+1$ such that ${t-1 \choose m-1} \equiv 0 \pmod{q}$ for all integers $m(q-1)+1 \leq t \leq T$.  The integer $T$ is well-defined by assumption, and we have $\Up{q}{m} = T+1$.  For the sake of contradiction, we assume that $T \geq mq$.  By definition, we have ${T-1 \choose m-1} \equiv \ldots \equiv {m(q-1) \choose m-1} \equiv 0 \pmod{p}$.  By Pascal's Rule, we obtain ${T-2 \choose m-2}\equiv \ldots \equiv {m(q-1) \choose m-2} \equiv 0 \pmod{p}$.  We may iterate the application of Pascal's Rule $m-1$ times to obtain $0 \equiv {m(q-1) \choose m-1} \equiv {m(q-1) \choose m-2} \equiv \ldots \equiv {m(q-1) \choose 1} \equiv {m(q-1) \choose 0} \equiv 1 \pmod{p}$, a contradiction.\\
% The proof is saying that the Sierpinski triangle (i.e. Pascal's triangle modulo p) has zero-entry sub-triangles of limited size.  The coefficients we are looking at are on a diagonal of Pascal's triangle.  Knowing values along a diagonal and using Pascal's rule helps determine the rest of the sub-triangle.

\item By the definition of $\Up{q}{m}$, we must show that ${m(q-1)+1-1 \choose m-1} \not \equiv 0 \pmod p$.  We use Equation \ref{eqn:Kummer2} to show $\nu_{p}({m(q-1) \choose m-1})=0$.

Note that the base-$p$ expansion of $mq-m$ is $(m-1)p^{\alpha} + (p-1)p^{\alpha -1} +\cdots + (p-1)p+(p-m)$.  The base-$p$ expansion of $m-1$ is $(m-1)$ since $p \geq 2m-1$.  Subtracting, we find the base-$p$ expansion of $m(q-1)-(m-1)$ is  $(m-1)p^{\alpha}+(p-1)p^{\alpha -1 }+ \cdots + (p-1)p+(p-2m+1)$.  By Equation \ref{eqn:Kummer2}
\begin{flalign*}
\nu_p({m(q-1) \choose m-1}) &=  \frac{s_p(m-1)+s_p(m(q-1)-(m-1))-s_p(m(q-1))}{p-1}\\
					&=  \frac{(m-1)+[(m-1)+(p-1)(\alpha -1)+(p-2m+1)]}{p-1}\\
					&~~~~~~ -\frac{[(m-1)+(p-1)(\alpha -1)+(p-m)]}{p-1}\\
					&=  0.&&
\end{flalign*}

\item We begin by noting that the difference between the claimed value and the smallest $\Up{q}{m}$ allowed by the definition is $2m-p-1$.  Thus, by the definition of $\Up{q}{m}$, we must show that ${mq-p+m-1 \choose m-1} \not \equiv 0 \pmod p$ and that ${mq-p+m-1-j \choose m-1} \equiv 0 \pmod p$ for $1 \leq j \leq 2m-p-1$.  

We use Equation \ref{eqn:Kummer2} to first show $\nu_{p}({mq-p+m-1 \choose m-1})=0$.  Note that the base-$p$ expansion of $mq-p+m-1$ is $(m-1)p^{\alpha}+(p-1)p^{\alpha -1}+\cdots +(p-1)p+(m-1)$.  The base-$p$ expansion of $m-1$ is $(m-1)$ since $m \leq p$.  Subtracting, the base-$p$ expansion of $mq-p+m-1-(m-1)=mq-p$ is $(m-1)p^{\alpha}+(p-1)p^{\alpha -1}+ \cdots +(p-1)p+0$.  By Equation \ref{eqn:Kummer2}

\begin{flalign*}
&\nu_p({mq-p+m-1 \choose m-1})\\
&= \frac{s_p(m-1)+s_p(mq-p)-s_p(mq-p+m-1)}{p-1}\\
					&= \frac{(m-1)+[(m-1)+(p-1)(\alpha -1)]-[2(m-1)+(p-1)(\alpha -1)]}{p-1}\\
					&= 0.&&
\end{flalign*}

We now use Equation \ref{eqn:Kummer2} to show $\nu_{p}({mq-p+m-1-j \choose m-1}) \neq 0$ for $1 \leq j \leq 2m-p-1$.  First note that since $m \leq p$, we have $j \leq m-1$.   Note that the base-$p$ expansion of $mq-p+m-1-j$ is $(m-1)p^{\alpha}+(p-1)p^{\alpha -1}+\cdots +(p-1)p+(m-1)-j$.  The base-$p$ expansion of $m-1$ is $(m-1)$ since $m \leq p$.  Subtracting, the base-$p$ expansion of $mq-p+m-1-j-(m-1)=mq-p-j$ is $(m-1)p^{\alpha}+(p-1)p^{\alpha -1}+ \cdots +(p-2)p+(p-j)$.  By Equation \ref{eqn:Kummer2}

\begin{flalign*}
&\nu_p({mq-p+m-1-j \choose m-1})\\
				 &= \frac{s_p(m-1)+s_p(mq-p-j)-s_p(mq-p+m-1-j)}{p-1}\\
					&= \frac{(m-1)+[(m-1)+(p-1)(\alpha -1)-1+(p-j)]}{p-1}\\
					&~~~~~~~-\frac{[2(m-1)+(p-1)(\alpha -1)-j]}{p-1}\\
					&= 1.&&
\end{flalign*}

\item Consider ${t-1 \choose m-1}$ as a polynomial in $\Z_p[t]$.  Since $${t-1 \choose m-1} = \frac{(t-1)(t-2)\ldots(t-(m-1))}{(m-1)!},$$  this polynomial clearly is of degree $m-1$ with roots $1, 2 \ldots, m-1$.
\end{enumerate}
\end{proof}

We now give the value of $\Low{q}{m}$ in the case that $q$ and $m$ are powers of the same prime $p$.

\begin{prop}\label{prop:Low}
For a prime $p$ and integers $r$ and $s$, we have $\Low{p^r}{p^s} = p^{r+s}.$
\end{prop}

\begin{proof} By definition, we must show that the smallest integer $t \geq p^s+1$ for which ${t \choose p^s} \equiv 0 \pmod{p^r}$ is $t=p^{r+s}$.  We must show that $\nu_{p}({t \choose p^s}) < r$ for $p^s+1 \leq t < p^{r+s}$ and that $\nu_{p}({p^{r+s} \choose p^s}) = r$.

We first show that $\nu_{p}({p^{r+s} \choose p^s}) = r$.  Note that the base-$p$ expansion of $p^{r+s}$ is $1p^{r+s}+0p^{r+s-1}+\cdots +0p+0$ and the base-$p$ expansion of $p^s$ is $1p^s+0p^{s-1}+\cdots +0p+0$.  Subtracting, the base-$p$ expansion of $p^{r+s}-p^s$ is  $(p-1)p^{r+s-1}+\cdots +(p-1)p^s+0p^{s-1}+\cdots +0p+0$.  By Equation \ref{eqn:Kummer2},

\begin{eqnarray*}
\nu_p({p^{r+s} \choose p^s}) & = &\frac{s_p(p^s)+s_p(p^{r+s}-p^s)-s_p(p^{r+s})}{p-1}\\
					& = & \frac{1+r(p-1)-1}{p-1}\\
					& = &r.
\end{eqnarray*}

We now show that $\nu_{p}({t \choose p^s}) < r$ for $p^s+1 \leq t < p^{r+s}$.  The base-$p$ expansion of $t$ is $t_{r+s-1}p^{r+s-1}+\cdots + t_1p+t_0$  and for $p^s$ is $1p^s+0p^{s-1}+\cdots +0p+0$.   Subtracting, the base-$p$ expansion of $t-p^s$ is $t'_{r+s-1}p^{r+s-1}+\cdots +t'_{s+1}p^{s+1}+t'_sp^s+t_{s-1}p^{s-1}+\cdots +t_1p+t_0$.  By Equation \ref{eqn:Kummer2}

\begin{eqnarray*}
\nu_{p}({t \choose p^s}) & = & \frac{s_p(p^s)+s_p(t-p^s)-s_p(t)}{p-1}\\
				& = & \frac{1+(t'_{r+s-1}+\cdots +t'_s+t_{s-1}+\cdots +t_1+t_0)-(t_{r+s-1}+\cdots +t_0)}{p-1}.
\end{eqnarray*}

After cancelling like terms, we obtain

\begin{eqnarray*}
\nu_{p}({t \choose p^s}) =  \frac{1 +(t'_{r+s-1}-t_{r+s-1})+\cdots +(t'_s-t_s)}{p-1}.
\end{eqnarray*}

By the rules of subtraction and as $t \geq p^s+1$, there exists an index $i$ with $s \leq i  \leq r+s-1$ for which $t'_i - t_i = -1$.  Thus,
\begin{eqnarray*}
  \nu_{p}({t \choose p^s})  &\leq & \frac{1-1+(r-1)(p-1)}{p-1} = r-1. \\
  \end{eqnarray*}
  
\end{proof}

\begin{thm}\label{thm:cyclicequality}
For integers $r$ and $s$, and a prime $p$, for $q=p^r$ we have $\D{\elem{p^s}{\bf x}}{p^s}{\Z_q}=p^{r+s}.$
\end{thm}

\begin{proof} We apply Theorem \ref{thm:cyclicbound}.  From Part 1 of Proposition \ref{prop:Up}, we have $\D{\elem{p^s}{\bf x}}{p^s}{\Z_q} \leq \Up{q}{p^s} \leq qp^s=p^{r+s}$.  From Proposition \ref{prop:Low}, we have $\D{\elem{p^s}{\bf x}}{p^s}{\Z_q} \geq \Low{q}{p^s}=p^{r+s}$.  Thus, equality holds.
\end{proof}

\section{More general lower and upper bounds}\label{section:pgroups}

We now provide a generalization of Theorem \ref{thm:cyclicbound} to products of the form $\pgroup{r}$.  We proceed in a similar way as in the set-up and proof of Theorem \ref{thm:cyclicbound}.

Define $\Up{\pvector{r}}{m}$ to be the smallest integer $t \geq m\sum_{i=1}^r(p^{\alpha_i}-1)+1$ such that 

$$\sum_{0 \leq 2j \leq m-1} {t \choose 2j} \not \equiv \sum_{1 \leq 2j+1 \leq m-1} {t \choose 2j+1}  \pmod{p}. $$
As before, an alternate definition of $\Up{\pvector{r}}{m}$ is the smallest integer  $t \geq m\sum_{i=1}^r(p^{\alpha_i}-1)+1$ such that ${t-1 \choose m-1} \not \equiv 0 \pmod{p}$.

\begin{thm}\label{thm:p-grouprank-rbound}
$\D{\elem{m}{\bf x}}{m}{\pgroup{r}} \leq {\Up{\pvector{r}}{m}}.$
\end{thm}

\begin{proof} 
Let $S=(a_1, \ldots ,a_{\ell})$ be a sequence over $\pgroup{r}$. For each $1 \le i \le \ell$, we write $a_i=[a_{i,1},\dots,a_{i,r}]$ where $a_{i,k} \in \mathbb{Z}_{p^{\alpha_k}}$ for all $1 \le k \le r$.
As opposed to seeking subsequences $S'$ of length at least $m$ such that $e_m(S')=0$, we may seek idempotent solutions that are not $m$-artificial to the following {\em system} of $r$ polynomial equations,
$$ \displaystyle\sum_{1 \le i_1 < \dots < i_m \le \ell} \displaystyle\prod^m_{j=1} a_{i_j,k}x_{i_j} \equiv 0 \pmod{p^{\alpha_k}} ~~ \forall k ~{\text{where}}~ 1 \leq k \leq r.$$
To prove the upper bound, consider this system of polynomial equations, each of which is of degree $m$, when the number of variables is $\Up{\pvector{r}}{m} \geq m\sum_{i=1}^r(p^{\alpha_i}-1)+1$.

Clearly, all $m$-artificial idempotent ${\Up{\pvector{r}}{m}}$-tuples are solutions to this system of equations since each monomial of each polynomial is a product of $m$ variables (and so at least one variable in each monomial evaluates as $0$ and so each monomial evaluates as  $0$).  From these solutions, we know that the number of even idempotent solutions $A$ is at least $\Even{\Up{\pvector{r}}{m}}{m}$ and the number of odd idempotent solutions $B$ is at least $\Odd{\Up{\pvector{r}}{m}}{m}$.  By the definition of $\Up{\pvector{r}}{m}$, we have that $$|\Even{\Up{\pvector{r}}{m}}{m}| \not \equiv |\Odd{\Up{\pvector{r}}{m}}{m}| \pmod{p}.$$  Thus, by Theorem \ref{thm:BS2}, there exists an idempotent solution that is not $m$-artificial.
\end{proof}

\begin{cor} \label{mD}
\begin{eqnarray*}
\D{\elem{m}{\bf x}}{m}{\pgroup{r}} & \leq & m (\displaystyle\sum_{i=1}^r(p^{\alpha_i}-1)+1)\\
							& =& m\Dav{\pgroup{r}}.
\end{eqnarray*}
\end{cor}

\begin{proof} 
The same argument as in the proof of Proposition \ref{prop:Up}.1 implies that 
\begin{eqnarray*}
{\Up{\pvector{r}}{m}} & \leq & (m\sum_{i=1}^r(p^{\alpha_i}-1)+1) + m-1= m  (\sum_{i=1}^r(p^{\alpha_i}-1)+1).
\end{eqnarray*}
The result then follows from Theorems \ref{thm:Olson} and \ref{thm:p-grouprank-rbound}.
\end{proof}

\begin{lem}\label{lem:eitherornm}
For every integer $z$ such that ${z \choose m} \equiv 0 \pmod {n}$, we have $z < m$ or $z \geq \Low{n}{m}.$
\end{lem}
\begin{proof} This follows immediately from the definition of $\Low{n}{m}$.\end{proof}

\begin{thm}\label{thm:directproductlowerbound}
Let $A$ be the ring $\abeliangroup{b}$.  We have $$\D{\elem{m}{\bf x}}{m}{A} \geq \sum_{j=1}^b \Low{n_j}{m}-(b-1)m.$$
\end{thm}

%We give two proofs of Theorem \ref{thm:directproductlowerbound}.  The first proof of this theorem reflects thinking that might be described as group-theoretic, whereas the second is more combinatorial in nature.

\begin{proof} We show that the following sequence $S$ is $m$-zero free over $\abeliangroup{b}$.  

Let $S$ be a sequence over $\abeliangroup{b}$ with elements $g_1=[1,1, \ldots,1]$ repeated $\Low{n_1}{m}-1$ times, $g_2=[0,1, \ldots ,1]$ repeated $\Low{n_2}{m}-m$ times, $\ldots , g_b=[0, \ldots, 0,1]$ repeated $\Low{n_b}{m}-m$ times.  (Notice that the number of times that $g_1$ is repeated is different in format from that of the other $g_i$s.)  The sequence $S$ has length $|S|=(\Low{n_1}{m}-m)+\ldots+(\Low{n_b}{m}-m)+m-1.$

For the sake of contradiction, let $S'$ be an $m$-zero subsequence of length at least $m$.  For $i=1, \ldots, b$, let $s_i \geq 0$ be the number of times $g_i$ appears in $S'$.  Since $S'$ is a subsequence, we have $s_1 \leq \Low{n_1}{m}-1, s_2 \leq \Low{n_2}{m}-m,\ldots , s_b \leq \Low{n_b}{m}-m$.  Since $S'$ is an $m$-zero sequence of length at least $m$, we have $|S'|=s_1+\ldots +s_b \geq m$.  We also have

\begin{eqnarray*}
{s_1 \choose m} &\equiv& 0 \pmod {n_1} \\
{s_1+s_2 \choose m}&\equiv& 0 \pmod {n_2}\\
\ldots \\
\ldots \\
{s_1+s_2+\ldots +s_b \choose m} &\equiv& 0 \pmod {n_b}.
\end{eqnarray*}

We now prove by induction on $i \in [0,b-1]$ that $S'$ has the following property: $P(i) :s_1+\ldots +s_{b-i} \geq m$.

We first establish the base case.  When $i=0$, we have that since $S'$ is an $m$-zero sequence of length at least $m$, $|S'|=s_1+\ldots +s_b \geq m.$ Now we establish the induction step and so we assume that $P(i)$ holds for some $i \in [0,b-2]$.  That is, assume that $s_1+\ldots+s_{b-i} \geq m$ for some $i \in [0,b-2]$.  Since ${s_1+\ldots +s_{b-i} \choose m} \equiv 0 \pmod {n_{b-i}},$ Lemma \ref{lem:eitherornm} implies that $s_1+ \ldots + s_{b-i} \geq \Low{n_{b-i}}{m}$.  If $s_1+\ldots +s_{b-(i+1)} < m,$ then 

\begin{eqnarray*}
\Low{n_{b-i}}{m} & \leq & s_1+ \ldots + s_{b-i}\\
			  &= &  (s_1+\ldots +s_{b-(i+1)})+s_{b-i}\\
			 & \leq & m-1+(\Low{n_{b-i}}{m}-m)\\
			 & =& \Low{n_{b-i}}{m}-1,
\end{eqnarray*} 
a contradiction.  Therefore, $s_1+\ldots +s_{b-(i+1)} \geq m$.

In particular, we have established that $s_1 \geq m$.  Since ${s_1 \choose m} \equiv 0 \pmod{n_1},$ Lemma \ref{lem:eitherornm} yields $s_1 \geq \Low{n_1}{m}$, contradicting the fact that the number $s_1$ of copies of $g_1$ contained in  $S'$ is at most $\Low{n_1}{m}-1$. \end{proof}

\begin{rem}
When $m=1$, Theorem \ref{thm:directproductlowerbound} recovers the well-known lower bound for the Davenport constant.  That is, $\Dav{\abegroup} \geq \sum_{i=1}^r(n_i-1)+1$, which Theorem \ref{thm:Olson} shows to be sharp for $p$-groups and groups of rank at most 2.  Theorem \ref{thm:directproductlowerbound} shows that it is sharp whenever $A$ is a product of the form $\pgroup{r}$ and $m$ is a power of $p$.
\end{rem}

\begin{rem}
We wish to emphasize that Theorem \ref{thm:directproductlowerbound} is for {\em any} direct product of cyclic groups.  That is, say, for example, we consider $\Z_6$, which is isomorphic to $\Z_2 \oplus \Z_3$.  The theorem applies to both representations and we may {\em choose} the one for which the theorem provides the best bound.
\end{rem}

\begin{thm}\label{thm:p-grouprank-rlowerbound}
For a prime $p$ and $s \geq 0,$ we have $$\D{\elem{p^s}{\bf x}}{p^s}{\pgroup{r}} = p^s\left(\sum_{i=1}^r(p^{\alpha_i}-1)+1\right)=p^s\Dav{\pgroup{r}}.$$
\end{thm}

%We give two proofs of Theorem \ref{thm:p-grouprank-rlowerbound}.  The first proof of this theorem reflects thinking that might be described as group-theoretic, whereas the second is more combinatorial in nature.

\begin{proof} 
By Corollary \ref{mD}, we have $$\D{\elem{p^s}{\bf x}}{p^s}{\pgroup{r}} \leq p^s \left(\sum_{i=1}^r(p^{\alpha_i}-1)+1\right).$$ 
On the other hand, Theorem \ref{thm:directproductlowerbound} gives $$\D{\elem{p^s}{\bf x}}{p^s}{\pgroup{r}} \geq p^s\left(\sum_{i=1}^r(p^{\alpha_i}-1)+1\right).$$
The last equality follows directly from Theorem \ref{thm:Olson}.
\end{proof}

\section{Improved bounds for $\D{\elem{m}{\bf x}}{m}{\Z_n}$ using the Girard-Newton formulae}\label{section:Girard-Newton}

We now state a historical set of relations between the elementary symmetric polynomials and the power sum polynomials.  These 17th-century relations are independently due to Albert Girard and Isaac Newton and known as the Girard-Newton formulae (or sometimes Newton's identities).

The symmetric functions that will be of interest to us consist of the following.  For $k \geq 0$, the {\it elementary symmetric polynomial of degree $k$} is the sum of all distinct products of $k$ distinct variables.  Thus, $\Elem{0}{n}=1, \Elem{1}{n}=x_1+\cdots +x_n, \Elem{2}{n}=\sum_{1 \leq i < j \leq n}x_ix_j$ and, so on, until,  $\Elem{n}{n} =x_1x_2\ldots x_n$.  The {\it $k$-th power sum polynomial} is $\Pow{k}{n} = \sum_{i=1}^{n} x_i^k$.

\begin{thm}[Girard-Newton formulae]
For all $n \geq 1$ and $1 \leq k \leq n$, we have
\begin{equation}\label{eqn:girard}
k\Elem{k}{n} = \sum_{i=1}^{k}(-1)^{i-1}\Elem{k-i}{n}\Pow{i}{n}.
\end{equation}

\end{thm}

We may rewrite Equations \ref{eqn:girard} in a manner that is independent of the number of variables, that is, we may rewrite Equations \ref{eqn:girard} in the ring of symmetric functions as

\begin{equation}\label{eqn:girardsimple}
ke_k =\sum_{i=1}^k (-1)^{i-1} e_{k-i}p_i.
\end{equation}

One may use the Girard-Newton formulae to recursively express elementary symmetric polynomials in terms of power sums as follows.

\begin{equation}\label{eqn:elementaryinpower}
e_k =(-1)^k \sum \prod_{i=1}^{k} \frac{(-p_i)^{j_i}}{j_i!i^{j_i}},
\end{equation}
where the sum extends over all solutions to $j_1+2j_2+\cdots+kj_k=k$ such that $ j_1, \ldots, j_k \geq 0$.  For example, we have $e_1=p_1, ~e_2=\frac{1}{2}p_1^2-\frac{1}{2}p_2, ~e_3=\frac{1}{6}p_1^3-\frac{1}{2}p_1p_2+\frac{1}{3}p_3,~e_4=\frac{1}{24}p_1^4-\frac{1}{4}p_1^2p_2+\frac{1}{8}p_2^2+\frac{1}{3}p_1p_3-\frac{1}{4}p_4$. Upon multiplying both sides of Equation \ref{eqn:elementaryinpower} by $k!$, we obtain on the right side {\it integer} coefficients.

Notice that for Equation \ref{eqn:elementaryinpower}, each term in the sum of the right side is a product that contains at most $k$ distinct power sum polynomials.  For a fixed $k$ we call a set $T$ of power sum polynomials a {\it dominating set for $e_k$} if each term in the sum contains at least one member of $T$. Let $\dom{k}$ denote the size of the smallest dominating set.  For $k=1$, the only dominating set is $\{p_1\}$, and so $t(1)=1$.  For $k=2$, the only dominating set is $\{p_1, p_2\}$, and so $t(2)=2$.  For $k=3$, any dominating set must contain both $p_1$ and $p_3$ and $\{p_1,p_3\}$ is a dominating set, and so $t(3)=2$.

\begin{lem}\label{lemma:dominatingsetsize}
We have $\dom{k}=\frac{k+2}{2}$ when $k$ is even, $\dom{k}=\frac{k+1}{2}$ when $k$ is odd.
\end{lem}

\begin{proof}
We may determine the size of the smallest dominating set by examining the solutions to the equation $j_1+2j_2+\cdots+kj_k=k$ such that $ j_1, \ldots, j_k \geq 0$.  There are solutions of the form (where we only specify the non-zero terms): $j_i=j_{k-i}=1$ for $1 \leq i \leq k/2-1$; $j_{k/2}=2$; and, $j_k=1$.  Thus selecting one element from each of the following sets $\{p_1, p_{k-1}\}, \ldots, \{p_{k/2-1},p_{k/2+1}\}, \{p_{k/2}\}, \{p_k\}$ is necessary to form a dominating set.  Thus, $\dom{k} \geq k/2+1$. Also, whenever $k/2+1 \leq i \leq k$ there is no solution with $j_i \geq 2$  and for any solution there is at most one index $i$, where $k/2 +1 \leq i \leq k$, so that $j_i=1$.  This implies that for any solution where for $k/2 +1 \leq i \leq k-1$ we have $j_i=1$, we also have $j_{i'}\geq 1$ for $1 \leq i' \leq k/2-1$.  Thus, $\{p_1, \ldots ,p_{k/2}, p_k\}$ is a dominating set of size $k/2+1$.  When $k$ is odd a similar argument allows us to claim that $\{p_1, \ldots ,p_{(k-1)/2},p_k\}$ is a minimum-sized dominating set. 

\end{proof}

%{\color{red} \tt Do we also wish to try to characterize all of the extremal sets in the above lemma?}

\begin{thm}\label{thm:m=2upperbound}
Let $n =2^{\nu_2(n)}m$ with  $m \geq 3$, and $b:=\lfloor \frac{\nu_2(n)-1}{2} \rfloor$.  We have
\begin{enumerate}
\item $\D{\elem{2}{\bf x}}{2}{\Z_n} \leq 2n-1$ when $\nu_2(n)=0$ (i.e. when $n$ is odd), and

\item $\D{\elem{2}{\bf x}}{2}{\Z_n} \leq (2+\frac{1}{2^b})n-1$ when $\nu_2(n) \geq 1$ (i.e. when $n$ is even).  
\end{enumerate}
\end{thm}

\begin{proof}
\begin{enumerate}

\item Let $n$ be an odd integer, $n \geq 3$.  Let $S$ be a sequence $(a_1, \ldots ,a_{2n-1})$ in $\Z_n$.  We may assume that all the elements in $S$ are non-zero.  Consider the following elements of $\Z_n \oplus \Z_n: [a_1,a_1^2], \ldots , [a_{2n-1},a_{2n-1}^2]$. Recall Theorem \ref{thm:Olson}, which gives here $\Dav{\Z_n \oplus \Z_n}=2n-1$.  That is, there exists a non-empty subset $J \subseteq \{1,\ldots, 2n-1\}$ such that $\sum_{j \in J} [a_j, a_j^2]=(0,0)$, and necessarily $|J| \geq 2$.  As a result, we have that $\sum_{j \in J}a_j \times \sum_{j \in J}a_j \equiv 0 \pmod{n}$ and that $\sum_{j \in J}a_j \times \sum_{j \in J}a_j= \sum_{j \in J} a_j^2 + 2\sum_{i \neq j,~i,j\in J}a_ia_j = 2\sum_{i \neq j,~i,j\in J}a_ia_j$.  Thus, $2\sum_{i \neq j,~i,j\in J}a_ia_j \equiv 0 \pmod{n}$, and as $n$ is odd we have $\sum_{i \neq j,~i,j\in J}a_ia_j \equiv 0 \pmod{n}$.  Thus, $\D{\elem{2}{\bf x}}{2}{\Z_n} \leq 2n-1$.

\item Let $n=2^{\nu_2(n)}m$ be an integer such that $\nu_2(n) \geq 1, m \geq 3$, and, for convenience, let $b:= \lfloor \frac{\nu_2(n)-1}{2} \rfloor \geq 0$.  Let $S$ be a sequence $(a_1, \ldots ,a_{(2+\frac{1}{2^b})n-1})$ in $\Z_n$.  We may assume that all the elements in $S$ are non-zero.  

{\sc Case: $\nu_2(n)$ is odd}

In this case, we have $\nu_2(n)=2b+1$.  Consider the following elements of $\Z_{m2^{\nu_2(n)-b}} \oplus \Z_{m2^{\nu_2(n)+1}} = \Z_{m2^{b+1}} \oplus \Z_{2n}= \Z_{\frac{n}{2^b}} \oplus \Z_{2n}:$  $$ [a_1,a_1^2], \ldots , [a_{(2+\frac{1}{2^b})n-1},a_{(2+\frac{1}{2^b})n-1}^2].$$ Recall Theorem \ref{thm:Olson}, which gives here $\Dav{\Z_{\frac{n}{2^b}} \oplus \Z_{2n}}=(2+\frac{1}{2^b})n-1$.  That is, there exists a non-empty subset $J \subseteq \{1,\ldots, (2+\frac{1}{2^b})n-1\}$ such that $\sum_{j \in J} [a_j, a_j^2]=(0,0)$, and necessarily $|J| \geq 2$.  That is, we have $\sum_{j \in J}a_j \equiv 0 \pmod{2^{b+1}m}$ and $\sum_{j \in J}a_j^2 \equiv 0 \pmod{2n}$.  As a result, we have that $\sum_{j \in J}a_j \times \sum_{j \in J}a_j \equiv 0 \pmod{2^{b+1}m} \times 0 \pmod{2^{b+1}m}\equiv 0 \pmod{2^{2b+2}m^2} \equiv 0 \pmod{2n}$ and that $\sum_{j \in J}a_j \times \sum_{j \in J}a_j= \sum_{j \in J} a_j^2 + 2\sum_{i \neq j,~i,j\in J}a_ia_j = 2\sum_{i \neq j,~i,j\in J}a_ia_j$.  Thus, $2\sum_{i \neq j,~i,j\in J}a_ia_j \equiv 0 \pmod{2n}$, and so we have $\sum_{i \neq j,~i,j\in J}a_ia_j \equiv 0 \pmod{n}$.  Thus, $\D{\elem{2}{\bf x}}{2}{\Z_n} \leq (2+\frac{1}{2^b})n-1$.

{\sc Case: $\nu_2(n)$ is even}

In this case, we have $\nu_2(n)=2b+2$.  Consider the following elements of $\Z_{m2^{\nu_2(n)-b}} \oplus \Z_{m2^{\nu_2(n)+1}} = \Z_{m2^{b+2}} \oplus \Z_{2n}= \Z_{\frac{n}{2^b}} \oplus \Z_{2n}:$  $$ [a_1,a_1^2], \ldots , [a_{(2+\frac{1}{2^b})n-1},a_{(2+\frac{1}{2^b})n-1}^2].$$ Recall Theorem \ref{thm:Olson}, which gives here $\Dav{\Z_{\frac{n}{2^b}} \oplus \Z_{2n}}=(2+\frac{1}{2^b})n-1$.  That is, there exists a non-empty subset $J \subseteq \{1,\ldots, (2+\frac{1}{2^b})n-1\}$ such that $\sum_{j \in J} [a_j, a_j^2]=(0,0)$, and necessarily $|J| \geq 2$.  That is, we have $\sum_{j \in J}a_j \equiv 0 \pmod{2^{b+2}m}$ and $\sum_{j \in J}a_j^2 \equiv 0 \pmod{2n}$. As a result, we have that $\sum_{j \in J}a_j \times \sum_{j \in J}a_j \equiv 0 \pmod{2^{b+2}m} \times 0 \pmod{2^{b+2}m} \equiv 0 \pmod{2^{2b+4}m^2} \equiv 0 \pmod{4n} \equiv 0 \pmod{2n}$ and that $$\sum_{j \in J}a_j \times \sum_{j \in J}a_j= \sum_{j \in J} a_j^2 + 2\sum_{i \neq j,~i,j\in J}a_ia_j = 2\sum_{i \neq j,~i,j\in J}a_ia_j.$$  Thus, $2\sum_{i \neq j,~i,j\in J}a_ia_j \equiv 0 \pmod{2n}$, and so we have $\sum_{i \neq j,~i,j\in J}a_ia_j \equiv 0 \pmod{n}$. Thus, $\D{\elem{2}{\bf x}}{2}{\Z_n} \leq (2+\frac{1}{2^b})n-1$.

\end{enumerate}
\end{proof}

\begin{rem}
Notice that Theorem \ref{thm:m=2upperbound} provides an upper bound on $\D{\elem{2}{\bf x}}{2}{\Z_n}$ for any $n \geq 3$, whereas Theorem \ref{thm:cyclicbound} and Proposition \ref{prop:Up} only provide upper bounds in the case that $n$ is a prime power.
\end{rem}

\begin{thm}\label{thm:mupperbound}
Let $n \geq 2$ and $\gcd{(n,m!)}=1$.  We have $\D{\elem{m}{\bf x}}{m}{\Z_n}\leq \Dav{\Z_n^{\dom{m}}}+m-1$.
\end{thm}

\begin{proof}
Let $n \geq 2$ be an integer such that $\gcd{(n,m!)}=1$, and let $M=\Dav{\Z_n^{\dom{m}}}+m-1$. Let $S$ be a sequence $(a_1, \ldots ,a_M)$ over $\Z_n$.\\
\noindent {\sc Case: $m$ is even}. Recall that for $m$ even, we have that $\{p_1, \ldots ,p_{m/2}, p_m\}$ is a minimum size dominating set for $e_m$. 
Consider the following $M$ elements of $\Z_n^{\dom{m}}$: 
$$[a_1^1, a_1^2, \ldots ,a_1^{m/2},a_1^m], [a_2^1, a_2^2, \ldots ,a_2^{m/2},a_2^m], \ldots , [a_M^1, a_M^2, \ldots ,a_M^{m/2},a_M^m].$$ 
Obviously, $M=\Dav{\Z_n^{\dom{m}}}+m-1 \geq \Dav{\Z_n^{\dom{m}}}$. 
That is, there exists a non-empty subset $J \subseteq \{1,\ldots, M\}$ such that 
$$\sum_{j \in J} [a_j, \ldots ,a_j^{m/2},a_j^m]=\overbrace{[0,\ldots, 0]}^{\dom{m}}.$$ 
Now, choose $J$ such that $|J|$ is largest. From Equation \ref{eqn:elementaryinpower}, we have that $m!e_m$ may be written as a sum of terms where each term is a product that contains at least one element from the above dominating set and each term has an integer coefficient. 
As a result, we may conclude that $e_m$ evaluates to zero modulo $n$ over $J$. 
If $|J| \geq m$, we are done. 
So, assume $|J| \leq m-1$, and consider the complement of $J$, $J^c= \{1,\ldots, M\} \setminus J$. 
By the assumptions, $|J^c| \geq \Dav{\Z_n^{\dom{m}}}$. 
Therefore, there exists a non-empty subset $J' \subseteq J^c$ such that 
$$\sum_{j \in J'} [a_j, \ldots ,a_j^{m/2},a_j^m]=\overbrace{[0,\ldots, 0]}^{\dom{m}}.$$ As a result, we may consider $J \cup J'$ which has size strictly larger than $J$ and the same property as $J$, contradicting the choice made above.

\noindent {\sc Case: $m$ is odd}. We proceed in a similar manner to the previous case, save that we have that $\{p_1, \ldots ,p_{(m-1)/2},p_m\}$ is a minimum dominating set for $e_m$.
\end{proof}

\section{Concluding remarks and open problems}\label{section:conclusion}

The most natural candidate for further research is the case of $\D{\elem{2}{\bf x}}{2}{\Z_n}$, which in our opinion preserves the same combinatorial and number-theoretic flavor of the $m=1$ case.
 
It is already known from Theorem \ref{thm:cyclicequality} that $\D{\elem{2}{\bf x}}{2}{\Z_{2^r}} = 2^{r+1}$.  Further upper bounds are obtained in Theorem \ref{thm:m=2upperbound}.
 
We have computed $\D{\elem{2}{\bf x}}{2}{\Z_n}$ for $2 \leq n \leq 16$ and $n=18$.  The results are presented in the list of following pairs  $(n, \D{\elem{2}{\bf x}}{2}{\Z_n})$ :  $(2,4), (3,5), (4,8), (5,6), (6,7), (7,10),$  $(8,16), (9,9), (10,9),$  $(11,13), (12,12), (13,14), (14,13), (15,12), (16,32)$ and $(18,13)$. 
 
We arrange this list as such:
\begin{enumerate} 
 
\item $(2,4), (4,8), (8,16), (16,32)$:  this is the case that $n$ is a power of $2$ and is already known from Theorem \ref{thm:cyclicequality} that $\D{\elem{2}{\bf x}}{2}{\Z_{2^r}} = 2^{r+1}$;
 
\item $(9,9), (12,12)$: in this case we have $\D{\elem{2}{\bf x}}{2}{\Z_n}=n.$
 
\begin{prob} For which $n$ does $\D{\elem{2}{\bf x}}{2}{\Z_n}=n$ hold?
 \end{prob}
 
\item $(10,9), (14,13), (15,12), (18,13)$: in this case we have $\D{\elem{2}{\bf x}}{2}{\Z_n}<n.$

\begin{prob} For which $n$ does $\D{\elem{2}{\bf x}}{2}{\Z_n} < n$ hold?
 \end{prob}

\item $(3,5), (5,6), (7,10), (11,13), (13,14)$: this is the case when $n$ is a prime.

We claim that $\D{\elem{2}{\bf x}}{2}{\Z_p} \geq p+1$.  For $p = 2$, this is established by Theorem \ref{thm:cyclicequality}.   For $p \geq 3$ consider the following sequence of length $p$: $S=(1, \ldots, 1, \frac{p+1}{2})$.  Any $2$-zero subsequence must contain at least two elements.  For a subsequence $S'$ we have $e_2(S')=\frac{j(j-1)}{2}+j\frac{p+1}{2} \equiv \frac{j}{2}(j+p) \not \equiv 0 \pmod{p}$ where $j$ counts the number of $1$'s in $S'$ for $1 \leq j \leq p-1$.  Thus, no $2$-zero subsequence exists.

 \begin{prob} For a prime $p$ with $p \equiv 1 \pmod{4}$, does $\D{\elem{2}{\bf x}}{2}{\Z_p}=p+1$ hold?
 \end{prob}
 
 \end{enumerate}

We claim that $\D{\elem{2}{\bf x}}{2}{\Z_p} \geq p+2$ for $p \equiv 3 \pmod 4$. For $p \geq 3$, consider the following sequence of length $p+1$: $S=(1, \ldots, 1, \frac{p+1}{2}, \frac{p+1}{2})$.  From the above, the only case that we need to consider is when $S'$ contains two copies of $\frac{p+1}{2}$.  For any such subsequence $S'$ we have 
\begin{eqnarray*}
e_2(S') &=& \frac{j(j-1)}{2}+2j\frac{p+1}{2}+\frac{(p+1)^2}{4}=\frac{j^2+(j+1)^2+p(4j+p+2)}{4},
\end{eqnarray*}
 where $j$ counts the number of $1$'s in $S'$ and $1 \leq j \leq p-1$.  Note that $$j^2+(j+1)^2+p(4j+p+2) \equiv j^2+(j+1)^2 \pmod{p}.$$  However, a prime is expressible as the sum of two squares if and only if congruent to $1 \pmod{4}$, a fact first observed by A.~Girard in 1625 (and later by P. de Fermat).

\begin{prob}

Given a prime $p$, let $q=p^r$ and $m=p^s$ be two powers of $p$. Is it true that every $m$-zero free sequence $S$ of length $\D{\elem{m}{\bf x}}{m}{\Z_q}-1 =mq-1$ over $\Z_q$ has the form $S=(a,\dots,a)$, where $a$ generates the additive group of $\Z_q$?

\end{prob}

\begin{prob}
Determine an upper bound for $\D{\elem{m}{\bf x}}{m}{\abeliangroup{b}}$.
\end{prob}

\begin{prob}
In light of Theorem \ref{thm:mupperbound}: determine an upper bound for $\D{\elem{m}{\bf x}}{m}{\Z_n}$ when $\gcd{(n,m!)} >1$.
\end{prob}

\noindent {\bf Acknowledgment.}  We give thanks to BIRS-CMO 2019 and Casa Matem\'atica Oaxaca, Mexico for supporting and hosting the event Zero-Sum Ramsey Theory: Graphs, Sequences and More 19w5132.

\end{document}